\documentclass[10pt]{amsart}
\usepackage{a4}
\usepackage{amssymb}
\usepackage{array}
\usepackage{booktabs}
\usepackage{multirow}
\usepackage{hhline}
\usepackage{mathabx}
\usepackage{color}
\usepackage{times}
\usepackage{graphicx}
\usepackage{pst-node}
\usepackage{pst-plot}
\usepackage[all]{xy}

\newtheorem{theorem}{Theorem}
\newtheorem{notation}{Notation}

\newtheorem{corollary}[theorem]{Corollary}
\newtheorem{lemma}[theorem]{Lemma}
\newtheorem{proposition}[theorem]{Proposition}

\newtheorem{example}{Example}

\newtheorem{remark}[theorem]{Remark}

\numberwithin{theorem}{section}
\numberwithin{equation}{section}

\newcommand{\norm}[1]{\left\Vert#1\right\Vert}

\newcommand{\la}{\langle}
\newcommand{\ra}{\rangle}

\newcommand{\Comp}{\mathbb{C}}

\newcommand{\F}{\mathcal{F}}
\newcommand{\g}{\mathbb{G}}

\newcommand{\n}{\mathbb{N}}

\newcommand{\tor}{\mathbb{T}}

\newcommand{\z}{\mathbb{Z}}

\global\long\def\tp{\mathop{\xymatrix{*+<.7ex>[o][F-]{\scriptstyle \top}}
 } }

\begin{document}
\title{Sobolev embedding properties on compact matrix quantum groups of Kac type}

\author{Sang-Gyun Youn}

\address{Sang-Gyun Youn : Department of Mathematics and Statistics, Queen's University, Jeffery Hall, Kingston, Ontario, K7L 3N6 Canada}
\email{s.youn@queensu.ca}

\thanks{This research is supported by the Natural Sciences and Engineering Research Council of Canada.}

\begin{abstract}
We establish sharp Sobolev embedding properties within a broad class of compact matrix quantum groups of Kac type under the polynomial growth or the rapid decay property of their duals. Main examples are duals of polynomially growing discrete quantum groups, duals of free groups and free quantum groups $O_N^+,S_N^+$. In addition, we generalize sharpend Hausdorff-Young inequalities, compute degrees of the rapid decay property for $\widehat{O_N^+},\widehat{S_N^+}$ and prove sharpness of Hardy-Littlewood inequalities on duals of free groups.
\end{abstract}

\maketitle

\section{Introduction}

It is a long tradition to study Fourier multipliers in harmonic analysis, and $L^p-L^q$ multipliers have played major roles in the theory of partial differential equations. A representative example is the Sobolev embedding property. More precisely, the Hardy-Littlewood-Sobolev theorem on tori states that
\begin{equation}
\norm{(1-\Delta)^{-\frac{d}{2}(\frac{1}{p}-\frac{1}{q})}(f)}_{L^q(\tor^d)}\lesssim \norm{f}_{L^p(\tor^d)}
\end{equation}
for all $1<p<q<\infty$ and $f\in L^p(\tor^d)$, where $\Delta$ is the Laplacian operator.

Sobolev embedding properties have been explored in a broad class including Lie groups \cite{Fo75,FiRu17,Va88,BrPeTaVa18}, and more generally, $L^p-L^q$ multipliers on Lie groups have been extensively studied \cite{CoGiMe93,AkRu15,AkNuRu15,AkNuRu16}. In particular, for connected compact Lie groups $G$, it is known that for any $1<p\leq 2$

\begin{equation}\label{Sobolev-Lie}
\left (\sum_{\pi\in \mathrm{Irr}(G)}\frac{n_{\pi}}{(1+\kappa_{\pi})^{n(\frac{1}{p}- \frac{1}{2} )}}\norm{\widehat{f}(\pi)}_{HS}^2 \right )^{\frac{1}{2}} = \norm{(1-\Delta)^{-\frac{n}{2}(\frac{1}{p}-\frac{1}{2})}(f)}_{L^2(G)}\lesssim \norm{f}_{L^p(G)},
\end{equation} 
where $n$ is the real dimension of $G$, $\Delta:\pi_{i,j}\mapsto -\kappa_{\pi}\pi_{i,j}$ is the Laplacian operator and $\norm{A}_{HS}=\mathrm{tr}(A^*A)^{\frac{1}{2}}$. Since the natural length function $|\cdot |$ on $\mathrm{Irr}(G)$ satisfies $|\pi|\approx \kappa_{\pi}^{\frac{1}{2}}$ \cite[Lemma 5.6.6]{Wa73}, the above (\ref{Sobolev-Lie}) is equivalent to
\begin{equation}\label{Sobolev-Lie2}
\left (\sum_{\pi\in \mathrm{Irr}(G)}\frac{n_{\pi}}{(1+|\pi|)^{n(\frac{2}{p}-1)}}\norm{\widehat{f}(\pi)}_{HS}^2 \right )^{\frac{1}{2}} \lesssim \norm{f}_{L^p(G)}.
\end{equation}

Note that (\ref{Sobolev-Lie2}) detects the real dimension of $G$, which is an important geometric quantity. The main purpose of this study is to generalize (\ref{Sobolev-Lie2}) to the framework of {\it compact quantum groups} by employing geometric information of the underlying quantum group such as growth rates and the rapid decay property. Indeed, for highly important examples, we will show that the polynomial growth order or the degree of rapid decay property replaces the role of the real dimension $n$.

The theory needed to explore Sobolev embedding properties for compact quantum groups is so-called {\it non-commutative $L^p$-analysis}. On quantum groups and quantum tori, $L^p-L^q$ multipliers have been studied from various perspectives \cite{JuPaPaPe17,GoJuPa17,FrHoLeUlZh17,AkMaRu18,XiXuYi18}. In particular, due to \cite[Theorem4.3]{AkMaRu18} which generalizes a theorem of H$\ddot{\mathrm{o}}$rmander, if $\g$ is a compact matrix quantum group of Kac type whose dual $\widehat{\g}$ has a polynomial growth 
\begin{equation}
 b_k=\sum_{\alpha\in \mathrm{Irr}(\g): |\alpha |\leq k}n_{\alpha}^2\leq C(1+k)^{\gamma}
\end{equation}
for some $C,\gamma>0$, then for any $1<p\leq 2$ we have
\begin{equation}\label{ineq-poly}
\left (\sum_{\alpha\in \mathrm{Irr}(\g)}\frac{n_{\alpha}}{(1+|\alpha |)^{ \gamma (\frac{2}{p}-1)}}\norm{\widehat{f}(\alpha)}_{HS}^2 \right )^{\frac{1}{2}}\lesssim \norm{f}_{L^p(\g)},
\end{equation}
where $|\cdot|$ is the natural length function on $\mathrm{Irr}(\g)$. To our best knowledge, if we exclude compact Lie groups and duals of polynomially growing discrete groups, it is not known whether the above inequalities (\ref{ineq-poly}) are sharp. However, we will show that (\ref{ineq-poly}) is sharp if one of the following natural assumptions holds:
\begin{enumerate}
\item (Corollary \ref{cor-poly-sharp}) if $b_k\approx (1+k)^{\gamma}$ and there exists a standard noncommutative semigroup $(T_t)_{t>0}$ on $L^{\infty}(\g)$ whose infinitesimal generator $L$ satisfies
\begin{equation}
L(u^{\alpha}_{i,j})=-l(\alpha)u^{\alpha}_{i,j}~\mathrm{with~}l(\alpha)\sim |\alpha|.
\end{equation}
\item (Corollary \ref{cor-poly}) if $b_k\lesssim (1+k)^{\gamma}$ and $\displaystyle s_k=\sum_{\alpha: |\alpha|=k} n_{\alpha}^2 \gtrsim (1+k)^{\gamma - 1}$.
\end{enumerate}

The above results establish sharp Sobolev embedding properties for connected compact Lie groups, duals of polynomially growing discrete groups, $O_2^+$ and $S_4^+$.

Despite the above strong conclusion under the polynomial growth of $\widehat{\g}$, it is important to focus on duals of free groups $\widehat{\mathbb{F}_{N}}$ and free quantum groups $O_N^+, S_N^+$ whose duals are exponentially growing. Arguably, these are the most important examples of compact quantum groups in view of operator algebras \cite{Wa95,VaWa96,Wa98,Vo11,Br12,VeVo13,Is15a,FrVe16,Br17b,Is17,BrVe18}, and non-commutative $L^p$-analysis on $\widehat{\mathbb{F}_N},O_N^+,S_N^+$ has been extensively studied \cite{JuMePa14b,MeRi17,MeDe17,JuPaPaPe17,Wa17,FrHoLeUlZh17,Yo18a,Yo18b}. From this viewpoint, one of the main aims of this paper is to establish the analogues of (\ref{ineq-poly}) for $\widehat{\mathbb{F}_N},O_N^+,S_N^+$ sharply.

To settle this, we will take two strategies. First of all, we extend \cite[Theorem 3.2]{Yo18b} to genral compact matrix quantum groups whose duals have the rapid decay property (Theorem \ref{thm-rapid}). We call it {\it sharpened Hausdorff-Young inequalities} and explain why such a phenomenon does not appear in the category of compact Lie groups, duals of discrete groups, $O_2^+$ and $SU_q(2)$ (Section \ref{sec-HY}). Then, by applying the complex interpolation method between the sharpened Hausdorff-Young inequalities (Theorem \ref{thm-rapid}) and Hardy-Littlewood inequalities \cite[Theorem 3.8]{Yo18a}, for any $1<p\leq 2$ we obtain
\begin{equation}\label{ineq-exp}
\left ( \sum_{\alpha\in \mathrm{Irr}(\g)}\frac{1}{(1+|\alpha|)^{3(\frac{2}{p}-1)}}n_{\alpha}\norm{\widehat{f}(\alpha)}_{HS}^2 \right )^{\frac{1}{2}}\lesssim \norm{f}_{L^p(\g)}
\end{equation}
under the assumption that $\g$ is one of $\widehat{\mathbb{F}_N}$, $O_{N+1}^+$ and $S_{N+3}^+$ with $N\geq 2$.

Secondly, the problem to check whether the exponent $3$ in (\ref{ineq-exp}) is optimal amounts to ultracontractivity problems of certain semigroups associated with free groups or free quantum groups due to \cite[Theorem 1.1]{Xi17}. More precisely, for the Poisson or heat semigroup $(T_t)_{t>0}$ of $\g=\widehat{\mathbb{F}_N},O_{N+1}^+$ or $S_{N+3}^+$, we will prove that there exists a universal constant $K>0$ such that
\begin{equation}
\norm{e^{-t}T_t(f)}_{L^{\infty}(\g)}\leq \frac{K \norm{f}_{L^2(\g)}}{t^{\frac{d}{2}}}~\mathrm{for~all~}f\in L^1(\g)~\mathrm{and~}t>0
\end{equation}
if and only if $d\geq 3$ (Corollary \ref{cor-ultra-sharp}). This confirms that (\ref{ineq-exp}) is sharp for $\widehat{\mathbb{F}_N}$, $O_{N+1}^+$ and $S_{N+3}^+$ with $N\geq 2$.

Lastly, we note that this study is applicable
\begin{enumerate}
\item (Corollary \ref{cor-rd}) to calculate the rapid decay degrees of $\widehat{O_N^+}$ and $\widehat{S_{N}^+}$ and 
\item (Corollary \ref{cor-application}) to prove sharpness of Hardy-Littlewood inequalities on $\widehat{\mathbb{F}_N} $ presented in \cite[Theorem 4.4]{Yo18a}. 
\end{enumerate}

\section{Preliminaries}\label{sec-Pre}

\subsection{Compact quantum groups and the representation theory}

A {\it compact quantum group} $\g$ is a pair $(C(\g),\Delta)$ where $C(\g)$ is a unital $C^*$-algebra and $\Delta:C(\g)\rightarrow C(\g)\otimes_{\mathrm{min}}C(\g)$ is a unital $*$-homomorphism such that 
\begin{enumerate}
\item $(\Delta\otimes \mathrm{id})\circ \Delta= (\mathrm{id}\otimes \Delta)\circ \Delta$.
\item $\mathrm{span}\left \{\Delta(a)(b\otimes 1):a,b\in C(\g)\right\}$ and $\mathrm{span}\left \{\Delta(a)(1\otimes b):a,b\in C(\g)\right\}$ are dense in $C(\g)\otimes_{\mathrm{min}}C(\g)$.
\end{enumerate}

For a compact quantum group $\g$, there exists a unique state $h$ satisfying 
\begin{equation}
(\mathrm{id}\otimes h)\circ \Delta = h(\cdot )1 =(h\otimes \mathrm{id})\circ \Delta.
\end{equation}
We call $h$ the {\it Haar state} and $\g$ is said to be of {\it Kac type} if $h$ is tracial.

A (finite dimensional) {\it unitary representation} of $\g$ is $u=(u_{i,j})_{i,j=1}^{n_u}\in M_{n_u}\otimes C(\g)$ such that
\begin{equation}
u^*u=\mathrm{Id}_{n_u}\otimes 1= uu^*~\mathrm{and}~\Delta(u_{i,j})=\sum_{k=1}^{n_u}u_{i,k}\otimes u_{k,j}~\mathrm{for~all~}1\leq i,j\leq n_u.
\end{equation}

We say that a unitary representation $u$ is {\it irreducible} if 
\begin{equation}
\left \{T\in M_{n_{u}}:(T\otimes 1)u=u(T\otimes 1)\right\}=\Comp \cdot \mathrm{Id}_{n_u}
\end{equation}
and denote by $\mathrm{Irr}(\g)\cong \left \{u^{\alpha}=(u^{\alpha}_{i,j})_{1\leq i,j\leq n_{\alpha}}\right\}_{\alpha\in \mathrm{Irr}(\g)}$ the set of irreducible unitary representations up to unitary equivalence. Then the space of polynomials 
\begin{equation}
\mathrm{Pol}(\g)=\mathrm{span}\left \{u^{\alpha}_{i,j}:\alpha\in \mathrm{Irr}(\g)~\mathrm{and}1\leq i,j\leq n_{\alpha}\right\}
\end{equation}
is a dense $*$-subalgebra of $C(\g)$ and the Haar state $h$ is faithful on $\mathrm{Pol}(\g)$.

Associated to a compact quantum group $\g$ is the discrete dual quantum group $\widehat{\g}=(\ell^{\infty}(\widehat{\g}),\widehat{\Delta},\widehat{h})$. Among the structures of $\widehat{\g}$, the underlying von Neumann algebra $\ell^{\infty}(\widehat{\g})$ is $\ell^{\infty}-\oplus_{\alpha\in \mathrm{Irr}(\g)}M_{n_{\alpha}}$ and, if $\g$ is of Kac type, $\widehat{h}$ is a normal semifinite faithful tracial weight on $\ell^{\infty}(\widehat{\g})$ given by
\begin{equation}
\widehat{h}(A)=\sum_{\alpha\in \mathrm{Irr}(\g)}n_{\alpha}\mathrm{tr}(A_{\alpha})~\mathrm{for~all~}A=(A_{\alpha})_{\alpha\in \mathrm{Irr}(\g)}\in \ell^{\infty}(\widehat{\g})_+.
\end{equation}

See \cite{Wo87a,Wo87b,KuVa00,KuVa03,Ti08} for more details of locally compact quantum groups.

\subsection{Non-commutative $L^p$-spaces}

Throughout this paper, we assume that $\g$ is a compact quantum group of Kac type. Since $h$ is faithful on $\mathrm{Pol}(\g)$, the space $\mathrm{Pol}(\g)$ is canonically embedded into $B(L^2(\g))$ where $L^2(\g)$ is the completion of $\mathrm{Pol}(\g)$ with respect to the inner product $\la f,g\ra_{L^2(\g)}=h(g^*f)$ for all $f,g\in \mathrm{Pol}(\g)$.

We define an associated von Neumann algebra $L^{\infty}(\g)=\mathrm{Pol}(\g)''$ in $B(L^2(\g))$, and then the Haar state $h$ extends to a normal faithful tracial state $h$ on $L^{\infty}(\g)$.

For any $1\leq p<\infty$, the {\it non-commutative $L^p$-space} is defined as the completion of $\mathrm{Pol}(\g)$ with respect to the norm structure $\norm{f}_{L^p(\g)}=h(|f|^p)^{\frac{1}{p}}$ for any $f\in \mathrm{Pol}(\g)$. Then it is well known that $(L^{\infty}(\g),L^1(\g))_{\frac{1}{p}}=L^p(\g)$ where $(\cdot,\cdot)_{\theta}$ is the complex interpolation space and that $L^{p'}(\g)=L^p(\g)^*$ for any $1\leq p<\infty$ under the dual bracket $\la f,g\ra_{L^p(\g),L^{p'}(\g)}=h(gf)$ for all $f,g\in \mathrm{Pol}(\g)$.

On the dual side, for any $1\leq p <\infty$, the non-commutative $\ell^p$-space is defined as 
\begin{equation}
\ell^p(\widehat{\g})=\left \{A\in \ell^{\infty}(\widehat{\g}):\sum_{\alpha\in \mathrm{Irr}(\g)}n_{\alpha}\mathrm{tr}(|A_{\alpha}|^p)<\infty \right\}
\end{equation}
and the natural norm structure is $\norm{A}_{\ell^p(\widehat{\g})}\displaystyle =\left (\sum_{\alpha\in \mathrm{Irr}(\g)}n_{\alpha}\mathrm{tr}(|A_{\alpha}|^p)\right )^{\frac{1}{p}}$ for all $A \in \ell^p(\widehat{\g})$. Then $(\ell^{\infty}(\widehat{\g}),\ell^1(\widehat{\g}))_{\frac{1}{p}}=\ell^p(\g)$ and $\ell^{p'}(\widehat{\g})=\ell^p(\widehat{\g})^*$ hold for any $1\leq p<\infty$. The dual bracket between $\ell^p(\widehat{\g})$ and $\ell^{p'}(\widehat{\g})$ is given by
\begin{equation}
\la A,B\ra_{\ell^p(\widehat{\g}),\ell^{p'}(\widehat{\g})}=\sum_{\alpha\in \mathrm{Irr}(\g)}n_{\alpha}\mathrm{tr}(B_{\alpha}A_{\alpha}).
\end{equation}

For the general theory of non-commutative $L^p$-spaces, see \cite{PiXu03,Pi03}.

\subsection{Compact matrix quantum groups and the rapid decay property}

The {\it tensor product representation} of two unitary representations $u$ and $v$ is 
\begin{equation}
u \tp v =(u_{i,j}v_{k,l})_{1\leq i,j\leq n_u, 1\leq k,l\leq n_v}\in M_{n_u}\otimes M_{n_v}\otimes C(\g)
\end{equation}
and every unitary representation is decomposed into a direct sum of irreducible unitary representations. If $\sigma$ is an irreducible component of $u\tp v$, we write $\sigma\subseteq u\tp v$.

A {\it compact matrix quantum group} is a compact quantum group $\g$ for which there exists a unitary representation $w$ such that every $u^{\alpha}\in \mathrm{Irr}(\g)$ is an irreducible component of $w^{\tiny \tp n}$ for some $n\in \left \{0\right\}\cup \n$. In this case, we can define a natural length function on $\mathrm{Irr}(\g)$ by 
\begin{equation}
|\alpha|=\min \left \{n\in \left \{0\right\}\cup \n : u^{\alpha}\subseteq w^{\tiny \tp n}\right\}~\mathrm{for~all~}\alpha\in \mathrm{Irr}(\g).
\end{equation}

Throughout this paper, we will use the following notations frequently.

\begin{notation}
Let $\g$ be a compact matrix quantum group and $k\in \left \{ 0\right\}\cup \n$.
\begin{enumerate}
\item $k$-balls are defined as $B_k=\left \{\alpha\in \mathrm{Irr}(\g):|\alpha|\leq k\right\}$ and $b_k$ is defined by $\displaystyle \sum_{\alpha\in B_k}n_{\alpha}^2$.
\item $k$-spheres $S_k$ are defined as $\left \{\alpha\in \mathrm{Irr}(\g):|\alpha|= k\right\}$ and $s_k$ is defined by $\displaystyle \sum_{\alpha\in S_k}n_{\alpha}^2$.
\item For each $\alpha\in \mathrm{Irr}(\g)$, we define $H_{\alpha}=\mathrm{span}\left \{ u^{\alpha}_{i,j}:1\leq i,j\leq n_{\alpha} \right\}$ and denote by $p_{\alpha}$ the orthogonal projection from $L^2(\g)$ onto $H_{\alpha}$.
\item We define $H_k=\oplus_{\alpha\in S_k} H_{\alpha}=\mathrm{span}\left \{u^{\alpha}_{i,j}:|\alpha|=k,1\leq i,j\leq n_{\alpha}\right\}$ and $p_{k}$ the orthogonal projection from $L^2(\g)$ onto $H_{k}$.
\end{enumerate}
\end{notation}

We say that $\widehat{\g}$ has a {\it polynomial growth} if there exists $C,\gamma>0$ such that $b_k\leq C(1+k)^{\gamma}$ \cite{BaVe09} and norm structures of $H_{\alpha}$ and $H_k$ are inherited from $L^2(\g)$.

In the sense of \cite{Ve07}, for a compact matrix quantum group $\g$, we say that its discrete dual $\widehat{\g}$ has the {\it rapid decay property} if there exists $C,\beta>0$ such that 
\begin{equation}\label{ineq-rapid}
\norm{f}_{L^{\infty}(\g)}\leq C(1+k)^{\beta}\norm{f}_{L^2(\g)}~\mathrm{for~all~}f\in H_k.
\end{equation}

\begin{notation}
We say that the discrete dual $\widehat{\g}$ of a compact matrix quantum group $\g$ has the rapid decay property with $r_k\lesssim  (1+k)^{\beta}$ if the above inequality (\ref{ineq-rapid}) holds.
\end{notation}

\subsection{Fourier analysis on compact quantum groups}

Within the framework of compact quantum groups, the Fourier transform $\mathcal{F}:L^1(\g)\rightarrow \ell^{\infty}(\widehat{\g})$, $\phi \mapsto \widehat{\phi}=(\widehat{\phi}(\alpha))_{\alpha\in \mathrm{Irr}(\g)}$, is defined by
\begin{equation}
\widehat{\phi}(\alpha)_{i,j}=\phi((u^{\alpha}_{j,i})^*)~\mathrm{for~all~}1\leq i,j\leq n_{\alpha}
\end{equation}
under the identification $L^1(\g)=L^{\infty}(\g)_*$. 

If $\g$ is of Kac type, we call $\displaystyle \sum_{\alpha\in \mathrm{Irr}(\g)}n_{\alpha}\mathrm{tr}(\widehat{\phi}(\alpha)u^{\alpha})=\sum_{\alpha\in \mathrm{Irr}(\g)}\sum_{i,j=1}^{n_{\alpha}}n_{\alpha}\widehat{\phi}(\alpha)_{i,j}u^{\alpha}_{j,i}$ the Fourier series of $\phi\in L^1(\g)$ and denote it by $\displaystyle \phi\sim \sum_{\alpha\in \mathrm{Irr}(\g)}n_{\alpha}\mathrm{tr}(\widehat{\phi}(\alpha)u^{\alpha})$. Indeed, equality $f=\displaystyle \sum_{\alpha\in \mathrm{Irr}(\g)}n_{\alpha}\mathrm{tr}(\widehat{f}(\alpha)u^{\alpha})$ holds for all $f\in \mathrm{Pol}(\g)$ since $\widehat{f}(\alpha)=0$ for all but finitely many $\alpha$.

It is known that the Fourier transform $\mathcal{F}:L^1(\g)\rightarrow \ell^{\infty}(\widehat{\g})$ is contractive, and the Plancherl theorem states that 
\begin{equation}
\norm{f}_{L^2(\g)}=\left ( \sum_{\alpha\in \mathrm{Irr}(\g)}n_{\alpha}\mathrm{tr}(\widehat{f}(\alpha)^*\widehat{f}(\alpha)) \right )^{\frac{1}{2}}~\mathrm{for~all~}f\in L^2(\g).
\end{equation}

Therefore, by the complex interpolation theorem, we obtain the Hausdorff-Young inequalities
\begin{equation}
\norm{\widehat{f}}_{\ell^{p'}(\widehat{\g})}\leq \norm{f}_{L^p(\g)}~\mathrm{for~all~}f\in L^p(\g)~\mathrm{and~}1\leq p\leq 2.
\end{equation}

\subsection{Complex interpolation on vector valued $\ell^p$-spaces}

In this section, we gather some well-known facts on complex interpolation methods extracted from \cite[Section 1]{Xu96}. For a family of Banach spaces $\left \{E_k\right\}_{k\in \z}$ with a positive measure $\mu$ on $\z$ we define vector valued $\ell^p$-spaces by
\begin{equation}
\ell^p(\left \{E_k\right\}_{k\in \z},\mu)=\left \{(x_k)_{k\in \z}:x_k\in E_k~\mathrm{for~all~}k\in \z~\mathrm{and~}\left (\norm{x_k}_{E_k}\right )_{k\in \z} \in \ell^p(\z, \mu)\right\}
\end{equation}
and the natural norm structure is 
\begin{equation}
\norm{(x_k)_{k\in \z}}_{\ell^p(\left \{E_k\right\}_{k\in \z},\mu)}= \left \{ \begin{array}{cc}\left ( \sum_{k\in \z}\norm{x_k}_{E_k}^p \mu(k) \right )^{\frac{1}{p}}&~\mathrm{if~}1\leq p<\infty\\ \sup_{k\in \z}\left \{ \norm{x_k}_{E_k}\right \}&~\mathrm{if~}p=\infty \end{array} \right .  .
\end{equation}

If $(E_k,F_k)$ is a compatible pair of Banach spaces for all $k\in \z$ and $\mu_0,\mu_1$ are two positive measures on $\z$, then for any $0<\theta<1$ we have
\begin{equation}\label{CI1}
\left (\ell^{p_0}(\left \{E_k\right\}_{k\in \z},\mu_0),\ell^{p_1}(\left \{F_k\right\}_{k\in \z},\mu_1) \right )_{\theta}= \ell^{p}(\left \{(E_k,F_k)_{\theta}\right\}_{k\in \z},\mu)
\end{equation}
with equal norm, where $\displaystyle \frac{1-\theta}{p_0}+\frac{\theta}{p_1}=\frac{1}{p}$ and $\mu=\displaystyle \mu_0^{\frac{p(1-\theta)}{p_0}}\mu_1^{\frac{p \theta}{p_1}}$. In particular, for $p_0=2=p_1$ and any $0<\theta<1$, we have
\begin{equation}\label{CI2}
\left (\ell^{2}(\left \{E_k\right\}_{k\in \z},\mu_0),\ell^{2}(\left \{E_k\right\}_{k\in \z},\mu_1) \right )_{\theta}= \ell^{2}(\left \{E_k\right\}_{k\in \z},\mu_0^{1-\theta}\mu_1^{\theta}).
\end{equation}

\subsection{Examples of compact matrix quantum groups}

\subsubsection{Duals of discrete groups}

Let $\Gamma$ be a discrete group and $C_r^*(\Gamma)$ be the associated reduced group $C^*$-algebra generated by left translation operators $\left \{\lambda_g\right\}_{g\in \Gamma}$. The unital $*$-homomorphism $\Delta:C_r^*(\g)\rightarrow C_r^*(\Gamma)\otimes_{\mathrm{min}}C_r^*(\Gamma),\lambda_g\mapsto \lambda_g\otimes \lambda_g$, determines a compact quantum group  $\widehat{\Gamma}=(C_r^*(\Gamma),\Delta)$ which we call the {\it dual of the discrete group} $\Gamma$. Then $\mathrm{Irr}(\widehat{\Gamma})=\left \{\lambda_g\right\}_{g\in \Gamma}$ and the Haar state is the vacuum state determined by $h:\lambda_g\mapsto \delta_{g,e}$. The associated von Neumann algebra $L^{\infty}(\widehat{\Gamma})$ is the group von Neumann algebra $VN(\Gamma)$ and $L^1(\widehat{\Gamma})=A(\Gamma)$ is called the Fourier algebra of $\Gamma$.

Moreover, if $S=\left \{g_j\right\}_{j=1}^n\subset \Gamma$ is a generating set, then $\displaystyle w= \sum_{j=1}^n e_{j,j}\otimes \lambda_{g_j}\in M_n\otimes C_r^*(\Gamma)$  makes $\widehat{\Gamma}$ into a compact matrix quantum group and
$\Gamma$ is said to be of polynomial growth if $b_k\lesssim (1+k)^{\gamma}$ for some $\gamma$ with respect to $w$. Moreover, if $\Gamma$ has a polynomial growth, there exists a non-negative integer $\gamma\in \left \{0\right\}\cup \n$ such that $b_k\approx (1+k)^{\gamma}$. We call the non-negative integer $\gamma$ the polynomial growth order of $\Gamma$.

If $\Gamma$ is a non-elementary hyperbolic group, it is known that $\Gamma~(=\widehat{\widehat{\Gamma}})$ has the rapid decay property with $r_k\lesssim 1+k$ \cite{Ha78,Ha88}.

\subsubsection{Free orthogonal quantum groups and free permutation quantum groups}

Let $N\geq 2$ and $C(O_N^+)$ be the universal unital $C^*$-algebra generated by $\left \{u_{i,j}\right\}_{i,j=1}^N$ satisfying
\begin{equation}
u_{i,j}^*=u_{i,j}~\mathrm{and~}\sum_{k=1}^N u_{k,i}u_{k,j}=\delta_{i,j}=\sum_{k=1}^N u_{i,k}u_{j,k}~\mathrm{for~all~}1\leq i,j\leq N.
\end{equation}
Then $\displaystyle \Delta:u_{i,j}\mapsto \displaystyle \sum_{k=1}^N u_{i,k}\otimes u_{k,j}$ extends to a unital $*$-homomorphism and $O_N^+=(C(O_N^+),\Delta)$ with $u=(u_{i,j})_{i,j=1}^N$ satisfies the axioms to be a compact matrix quantum group. We call $O_N^+$ {\it free orthogonal quantum groups} \cite{Wa95}. 

On the other hand, we denote by $C(S_N^+)$ the universal unital $C^*$-algebra generated by $\left \{u_{i,j}\right\}_{i,j=1}^N$ satisfying
\begin{equation}
u_{i,j}^2=u_{i,j}=u_{i,j}^*~\mathrm{and}\displaystyle \sum_{k=1}^N u_{i,k}=1=\sum_{k=1}^N u_{k,j}~\mathrm{for~all~}1\leq i,j\leq N.
\end{equation}
Then $\Delta:\displaystyle  u_{i,j}\mapsto \sum_{k=1}^N u_{i,k}\otimes u_{k,j}$ again extends to a unital $*$-homomorphism and $(C(S_N^+),\Delta)$ with $u=(u_{i,j})_{i,j=1}^N$  forms a compact matrix quantum group, which we call {\it free permutation quantum groups} $S_N^+$ \cite{Wa98}.

Let $\g$ be either $O_N^+$ or $S_{N+2}^+$. Then $\mathrm{Irr}(\g)=\left \{0\right\}\cup \n$, $\left \{ \begin{array}{cc}n_k\approx (\frac{N+\sqrt{N^2-4}}{2})^k\\ r_k\lesssim 1+k  \end{array} \right .$ if $N\geq 3$ and $\left \{ \begin{array}{cc} s_k= (1+k)^2 \\ b_k \approx (1+k)^3  \end{array} \right .$ if $N= 2$ \cite{Ve07,BaVe09}.

\subsection{Standard noncommutative semigroup and examples}

Let $\g$ be a compact quantum group of Kac type. We say that a semigroup $(T_t)_{t>0}$ is a {\it standard noncommutative semigroup} on $L^{\infty}(\g)$ if $(T_t)_{t>0}$ satisfies the following assumptions \cite{JuMe12,JiWa17}:
\begin{enumerate}
\item Every $T_t$ is a normal unital completely positive maps on $L^{\infty}(\g)$;
\item For any $t>0$ and $f,g\in L^{\infty}(\g)$, $h(T_t(f)g)=h(fT_t(g))$;
\item For any $f\in L^{\infty}(\g)$ $\displaystyle \lim_{t\rightarrow 0^+} T_t(f)=f$ in the strong operator topology.
\end{enumerate}

Note that the first two conditions imply $h(T_t(f))=h(f)$ for any $f\in L^{\infty}(\g)$ and that such a semigroup admits an infinitesimal generator $L$ such that $T_t=e^{tL}$.

\begin{example}\label{ex-1}
\begin{enumerate}
\item For any connected compact Lie groups, there exists the Poisson semigroup $(T_t)_{t>0}=e^{-t(-\Delta)^{\frac{1}{2}}}$ on $L^{\infty}(G)$ where $\Delta$ is the Laplacian operator \cite{St70}.
\item For duals of non-abelian free groups, there exists the {\it Poisson semigroup} $(T_t^F)_{t>0}$ on $VN(\mathbb{F}_N)$ such that $ \lambda_g \mapsto e^{-t|g|}\lambda_g$ \cite{Ha78,JuMePa14,MeDe17}.
\item For free orthogonal (resp. permutation) quantum groups $O_N^+$ (resp. $S_{N+2}^+$) with $N\geq 2$, there exists the heat semigroup $(T_t^O)_{t>0}$ (resp. $T_t^S$) on $L^{\infty}(O_N^+)$ (resp. $L^{\infty}(S_{N+2}^+)$) such that $u^k_{i,j}\mapsto e^{-t c_k} u^k_{i,j}$ with $c_k\sim \left \{ \begin{array}{cc}k&~\mathrm{if~}N\geq 3\\ k^2&\mathrm{~if~}N=2\end{array} \right .$ \cite[Lemma 1.7, Lemma 1.8 and Subsection 2.2]{FrHoLeUlZh17}.

\end{enumerate}
\end{example}

\section{Sharpness of Sobolev embedding properties for polynomially growing discrete quantum groups}\label{sec-poly}

Recall that, if $\g$ is a compact matrix quantum group of Kac type whose dual has the polynomial growth $b_k\lesssim (1+k)^{\gamma}$, then for any $1<p\leq 2$ we have
\begin{equation}\label{ineq-poly2}
\left (\sum_{\alpha\in \mathrm{Irr}(\g)}\frac{n_{\alpha}}{(1+|\alpha |)^{ \gamma (\frac{2}{p}-1)}}\norm{\widehat{f}(\alpha)}_{HS}^2 \right )^{\frac{1}{2}}\lesssim \norm{f}_{L^p(\g)}~\mathrm{for~all~}f\in L^p(\g).
\end{equation}

In this section, we will provide two sufficient conditions to prove sharpness of (\ref{ineq-poly2}). One strategy requires the existence of a standard noncommutative semigroup $(T_t)_{t>0}$ whose infinitesimal generator $L$ satisfies
\begin{equation}
L(u^{\alpha}_{i,j})=-l(\alpha)u^{\alpha}_{i,j}~\mathrm{with~}l(\alpha)\sim |\alpha|.
\end{equation}
and the other one depends on lower bounds of the growth of $k$-sphere $s_k$. These two strategies explain how we are able to obtain sharp Sobolev embedding properties for connected compact Lie groups and duals of polynomially growing discrete groups as already noted in \cite{Yo18a,Yo18c}. Furthermore, these methods also apply to the free orthogonal quantum group $O_2^+$ and the free permutation quantum group $S_4^+$.

\subsection{An approach using semigroups}

The purpose of this section is to extend some important techniques of \cite[Section 6]{Yo18a} to general compact matrix quantum groups of Kac type. Discussions in this section depend on the existence of a standard noncommutative semigroup whose infinitesimal generator behaves like the Poisson semigroup of connected compact Lie groups $G$.

Throughout this Section, let us suppose that there exists a standard semigroup $(T_t)_{t>0}$ whose infinitesimal generator $L$ satisfies
\begin{equation}
L(u^{\alpha}_{i,j})=-l(\alpha)u^{\alpha}_{i,j}~\mathrm{and~}l(\alpha)\sim |\alpha|~\mathrm{for~all~}\alpha\in \mathrm{Irr}(\g).
\end{equation}

Indeed, if $(P_t)_{t>0}$ is the Poisson semigroup of a connected compact Lie group $G$, then the associated infinitesimal generator satisfies
\begin{equation}
L(u^{\pi}_{i,j})=-\kappa_{\pi}^{\frac{1}{2}}u^{\pi}_{i,j}~\mathrm{and~}\kappa_{\pi}^{\frac{1}{2}}\sim |\pi|~\mathrm{for~all~}\pi\in \mathrm{Irr}(G).
\end{equation}

The following is a part of \cite[Theorem 1.1]{Xi17}, which is written in accordance with our notation.
\begin{theorem}[Theorem 1.1, \cite{Xi17}]\label{thm-Xiao}
Let $\g$ be a compact quantum group of Kac type and $(T_t)_{t>0}$ be a standard semigroup on $L^{\infty}(\g)$ with the infinitesimal generator $L$. Then, for the semigroup $(S_t)_{t>0}=(e^{-t}T_t)_{t>0}$ and $s>0$, the following are equivalent:
\begin{enumerate}
\item For any $1\leq p<q\leq \infty$, 
\[\norm{S_t(x)}_{L^{q}(\g)}\lesssim \frac{\norm{x}_{L^p(\g)}}{t^{s(\frac{1}{p}-\frac{1}{q})}}\mathrm{~ for~ all~} x\in L^{p}(\g)\mathrm{~ and~} t>0.\]
\item There exists $1\leq p<q\leq \infty$ such that
\[\norm{S_t(x)}_{L^{q}(\g)}\lesssim \frac{\norm{x}_{L^p(\g)}}{t^{s(\frac{1}{p}-\frac{1}{q})}}\mathrm{~ for~ all~} x\in L^{p}(\g)\mathrm{~ and~} t>0.\]
\item For any $1 < p<q < \infty$, 
\[\norm{(1-L)^{-s(\frac{1}{p}-\frac{1}{q})}(x)}_{L^{q}(\g)}\lesssim \norm{x}_{L^p(\g)}\mathrm{~ for~ all~} x\in L^{p}(\g).\]
\item There exists $1 < p<q < \infty$ such that 
\[\norm{(1-L)^{-s(\frac{1}{p}-\frac{1}{q})}(x)}_{L^{q}(\g)}\lesssim \norm{x}_{L^p(\g)}\mathrm{~ for~ all~} x\in L^{p}(\g).\]
\end{enumerate}
\end{theorem}

Using the above Theorem, we will find the optimal $s>0$ which validates the above four equivalent statements of Theorem \ref{thm-Xiao}. Indeed, by the following Theorem \ref{thm-poly-ultra}, we can find out that the optimal exponent $s$ is equal to the polynomial growth order $\gamma$ under the assumption $b_k\approx (1+k)^{\gamma}$.

Recall that a compact quantum group $\g$ is called {\it co-amenable} if there exists a contractive approximate identity $(e^i)_i$ in the convolution algebra $L^1(\g)$. Such a family $(e^i)_i$ satisfies $\displaystyle \lim_i \widehat{e^i}(\alpha)=\mathrm{Id}_{n_{\alpha}}$ for each $\alpha\in \mathrm{Irr}(\g)$. It is known that every dual of a polynomially growing discrete quantum group is co-amenable \cite[Proposition 2.1]{BaVe09}.

\begin{theorem}\label{thm-poly-ultra}
Let $\g$ be a general compact quantum group of Kac type and $w:\mathrm{Irr}(\g)\rightarrow (0,\infty )$ be a positive function.
\begin{enumerate}
\item If $\displaystyle C_w=\sum_{\alpha\in \mathrm{Irr}(\g)} e^{-2w(\alpha)}n_{\alpha}^2 <\infty$, then 
\begin{equation}
\left \| \sum_{\alpha\in \mathrm{Irr}(\g)} e^{-w(\alpha)}p_{\alpha}(f)\right \|_{L^{\infty}(\g)}\leq \sqrt{C_w} \left \|f\right \|_{L^2(\g)}~for~all~f\in L^2(\g).
\end{equation}
\item Conversely, if we assume that $\g$ is co-amenable and
\begin{equation}
\left \| \sum_{\alpha\in \mathrm{Irr}(\g) }  e^{-2w(\alpha)} p_{\alpha}(f)\right \|_{L^{\infty}(\g)}\leq  C \left \|f\right \|_{L^1(\g)}~for~all~f\in L^1(\g),
\end{equation}
then we have $\displaystyle \sum_{\alpha\in \mathrm{Irr}(\g)}e^{-2w(\alpha)}n_{\alpha}^2 \leq C$.
\end{enumerate}
\end{theorem}

\begin{proof}
\begin{enumerate}
\item For a given $\displaystyle f=\sum_{\alpha\in \mathrm{Irr}(\g)}p_{\alpha}(f)\in L^2(\g)$, we have
\begin{align}
\left \|p_{\alpha}(f)\right\|_{L^{\infty}(\g)}&\leq n_{\alpha}\mathrm{tr}(\left |\widehat{f}(\alpha)\right |)\leq n_{\alpha}^{\frac{3}{2}}\left \|\widehat{f}(\alpha) \right\|_{S^2_{n_{\alpha}}}= n_{\alpha}\norm{p_{\alpha}(f)}_{L^2(\g)}
\end{align}
by the Hausdorff-Young inequality and the H$\ddot{\mathrm{o}}$lder inequality, and hence
\begin{equation}
\norm{\sum_{\alpha\in \mathrm{Irr}(\g)} e^{-w(\alpha)}p_{\alpha}(f)}_{L^{\infty}(\g)}\leq \sum_{\alpha\in \mathrm{Irr}(\g)} n_{\alpha} e^{-w(\alpha)}\norm{p_{\alpha}(f)}_{L^2(\g)}\leq \sqrt{C_w} \norm{f}_{L^2(\g)}.
\end{equation}
\item Let $(e^i)_i$ be a contractive approximate identity in $L^1(\g)$ and then we may assume $e^i\in \mathrm{Pol}(\g)$ for all $i$. Then we have
\begin{align*}
C&\geq C \norm{e^i}_{L^1(\g)}\\
&\geq  \la \sum_{\alpha\in \mathrm{Irr}(\g)}  e^{-2w(\alpha)}  p_{\alpha}((e^i)^*),e^i \ra_{L^{\infty}(\g),L^1(\g)}\\
&=\sum_{\alpha\in \mathrm{Irr}(\g)}  e^{-2w(\alpha)}  n_{\alpha}\mathrm{tr}(e^i(\alpha)(e^i(\alpha))^*).
\end{align*}
Hence, by taking limit as $i\rightarrow \infty$, we obtain $\displaystyle \sum_{\alpha\in \mathrm{Irr}(\g)} e^{-2w(\alpha)}  n_{\alpha}^2\leq C$.
\end{enumerate}
\end{proof}

\begin{corollary}\label{cor-poly-ultra}
Let $\g$ be a compact matrix quantum group of Kac type whose dual satisfies $b_k\lesssim (1+k)^{\gamma}$ and $(T_t)_{t>0}$ be a standard semigroup whose infinitesimal generator $L$ satisfies
\begin{equation}
L(u^{\alpha}_{i,j})=-l(\alpha)u^{\alpha}_{i,j}~and~l(\alpha)\sim |\alpha|.
\end{equation}
Then there exists a constant $K=K(s)>0$ such that
\begin{equation}\label{ineq4}
\norm{e^{-t}T_t(f)}_{L^{\infty}(\g)}\leq  \frac{K \norm{f}_{L^2(\g)}}{t^{s}}~for~all~f\in L^2(\g)~and~t>0
\end{equation}
if $\displaystyle s\geq \frac{\gamma}{2}$. Moreover, the converse holds if $b_k\approx (1+k)^{\gamma}$.
\end{corollary}

\begin{proof}
\begin{enumerate}
\item Let us assume $s\geq \displaystyle \frac{\gamma}{2}$. Then, thanks to Theorem \ref{thm-poly-ultra}, it is sufficient to show that
\begin{equation}
\sup_{0<t<\infty} \left \{ t^{2s}\sum_{\alpha\in \mathrm{Irr}(\g)}\frac{n_{\alpha}^2}{e^{2t(1+l(\alpha))}} \right\}<\infty.
\end{equation}
Since there exists a constant $C>0$ such that $1+l(\alpha)\geq C(1+|\alpha|)$, we have
\begin{equation}
\sum_{\alpha\in \mathrm{Irr}(\g)}\frac{n_{\alpha}^2}{e^{2t(1+l(\alpha))}}\leq \sum_{\alpha\in \mathrm{Irr}(\g)}\frac{n_{\alpha}^2}{e^{2tC(1+|\alpha|)}}.
\end{equation}

Therefore, it is enough to prove that 
\begin{equation}
\sup_{0<t<\infty} \left \{ t^{2s}\sum_{\alpha\in \mathrm{Irr}(\g)}\frac{n_{\alpha}^2}{e^{2t(1+|\alpha |)}} \right\}<\infty.
\end{equation}

Note that 
\begin{align*}
\sum_{\alpha\in \mathrm{Irr}(\g)}\frac{n_{\alpha}^2}{e^{2t(1+|\alpha|)}}&=\lim_{n\rightarrow \infty}\sum_{k=0}^n\frac{s_k}{e^{2t(1+k)}}\\
&=\lim_{n\rightarrow \infty}\left \{  b_ne^{-2t(1+n)}+\sum_{k=0}^{n-1} b_k(e^{-2t(1+k)}-e^{-2t(2+k)}) \right\}\\
&\leq  \sum_{k=0}^{\infty} b_k\cdot 2t e^{-2t(1+k)} \lesssim t \sum_{k=0}^{\infty} (1+k)^{\gamma}e^{-2t(1+k)} .
\end{align*}

Since $\displaystyle x\mapsto x^{2s+2}e^{-2x}$ has an upper bound $C'=C'(s)$ on $[0,\infty )$, we have
\[\frac{t^{2s+1}(1+k)^{\gamma+2}}{e^{2t(1+k)}}\leq \frac{4}{\gamma}\cdot \frac{t^{2s+2}(1+k)^{2s+2}}{e^{2t(1+k)}} \leq \frac{ 4C'}{\gamma} ~\mathrm{for~any~}t\geq \frac{\gamma}{4} . \]
Therefore, $\displaystyle \sup_{\frac{\gamma}{4}\leq t<\infty} \left \{ t^{2s+1}\sum_{k=0}^{\infty} (1+k)^{\gamma}e^{-2t(1+k)} \right\} \leq \sum_{k=0}^{\infty}\frac{4C'}{\gamma (1+k)^2}<\infty$.

From now on, let us handle the case $\displaystyle 0<t<\frac{\gamma}{4}$. Since the function $f(x)=x^{\gamma}e^{-2x}$ is $\displaystyle \left \{ \begin{array}{cc}\mathrm{increasing}&\mathrm{~if~}x<\frac{\gamma}{2} \\ \mathrm{decreasing}&\mathrm{~if~}x>\frac{\gamma}{2} \end{array} \right . $, we have

\begin{align*}
t^{\gamma+1} \sum_{k=0}^{\infty} (1+k)^{\gamma} e^{-2t(1+k)}&\leq t \left (\sum_{k:(1+k)t<\frac{\gamma}{2}}+\sum_{k:(1+k)t\geq \frac{\gamma}{2}}\right ) \left [ (1+k)^{\gamma}t^{\gamma}e^{-2t(1+k)}\right ]\\
&\leq t[(\frac{\gamma}{2e})^{\gamma}\cdot \frac{\gamma}{2t}+\int_{\frac{\gamma}{2t}-2}^{\infty}(1+x)^{\gamma}t^{\gamma}e^{-2t(1+x)}dx]\\
&= t[(\frac{\gamma}{2e})^{\gamma}\cdot \frac{\gamma}{2t}+\int_{\frac{\gamma}{2}-t}^{\infty} y^{\gamma}e^{-2y}(\frac{dy}{t})]=(\frac{\gamma}{2e})^{\gamma}\cdot \frac{\gamma}{2}+\int_{\frac{\gamma}{2}-t}^{\infty} y^{\gamma}e^{-2y} dy.
\end{align*}
Hence, we have
\begin{align*}
\sup_{0<t<\frac{\gamma}{4}}\left \{t^{2s+1}\sum_{k=0}^{\infty} (1+k)^{\gamma}e^{-2t(1+k)} \right\}&\leq (\frac{\gamma}{4})^{2s-\gamma}\sup_{0<t<\frac{\gamma}{4}}\left \{t^{\gamma+1}\sum_{k=0}^{\infty} (1+k)^{\gamma}e^{-2t(1+k)} \right \} \\
& \leq (\frac{\gamma}{4})^{2s-\gamma}[ (\frac{\gamma}{2e})^{\gamma}\cdot \frac{\gamma}{2}+\int_{\frac{\gamma}{4}}^{\infty} y^{\gamma}e^{-2y} dy]<\infty .
\end{align*}
\item Conversely, let us assume $b_k\approx (1+k)^{\gamma}$ and (\ref{ineq4}) hold. Then 
\begin{equation}
\norm{e^{-t}T_t(f)}_{L^{\infty}(\g)}\leq \frac{K^2 2^{2s}\norm{f}_{L^1(\g)}}{t^{2s}},
\end{equation}
so that Theorem \ref{thm-poly-ultra} (2) tells us that
\begin{equation}
\sup_{0<t<\infty}\left \{ t^{2s} \sum_{\alpha\in \mathrm{Irr}(\g)}\frac{n_{\alpha}^2}{e^{t(1+|\alpha |)}}\right\}\leq K^2 2^{2s}<\infty
\end{equation}

Then, as in the proof of the if part, we have
\begin{align*}
K^2 2^{2s}\geq t^{2s} \sum_{k= 0}^{\infty}\frac{s_k}{e^{t(1+k)}}&=t^{2s} \sum_{k=0}^{\infty}b_k(e^{-t(1+k)}-e^{-t(2+k)})\\
&\gtrsim  t^{2s+1} \sum_{k=0}^{\infty} (1+k)^{\gamma}e^{-t(2+k)}\\
&=t^{2s-\gamma+1}e^{-t}\sum_{k=0}^{\infty}(1+k)^{\gamma}t^{\gamma}e^{-t(1+k)}\\
&\geq t^{2s-\gamma+1}e^{-t}\sum_{k: (1+k)t< \gamma }(1+k)^{\gamma}t^{\gamma}e^{-t(1+k)}.
\end{align*}

Since the function $x\mapsto x^{\gamma}e^{-x}$ is increasing in $[0,\gamma]$, we have
\begin{align*}
K^2 2^{2s}&\gtrsim  t^{2s-\gamma+1}e^{-t}\int_{0}^{\frac{\gamma}{t}-1} (1+x)^{\gamma}t^{\gamma}e^{-t(1+x)}dx\\
&=t^{2s-\gamma}e^{-t}\int_t^{\gamma}y^{\gamma}e^{-y}dy.
\end{align*}

By taking $t\rightarrow 0^+$, we obtain $\displaystyle \lim_{t\rightarrow 0^+}t^{2s-\gamma}<\infty$, so that $ 2s \geq \gamma$.
\end{enumerate}
\end{proof}

Therefore, thanks to Theorem \ref{thm-Xiao}, we obtain the following sharp Sobolev embedding property:

\begin{corollary}\label{cor-poly-sharp}
Let $\g$ be a compact matrix quantum group of Kac type whose dual satisfies $b_k \approx (1+k)^{\gamma}$ and let $(T_t)_{t>0}$ be a standard noncommutative semigroup whose infinitesimal generator $L$ satisfies
\begin{equation}
L(u^{\alpha}_{i,j})=-l(\alpha)u^{\alpha}_{i,j}~and~l(\alpha)\sim |\alpha|.
\end{equation}
Then the following are equivalent:
\begin{enumerate}
\item For any $1<p\leq 2$ there exists a constant $K=K(p)>0$ such that
\begin{equation}
\left (\sum_{\alpha\in \mathrm{Irr}(\g)}\frac{n_{\alpha}}{(1+|\alpha|)^{s (\frac{2}{p}-1)}}\norm{\widehat{f}(\alpha)}_{HS}^2\right )^{\frac{1}{2}}\leq K \norm{f}_{L^p(\g)}~for~all~f\in L^p(\g)
\end{equation}
\item There exist $1<p\leq 2$ and a constant $K>0$ such that
\begin{equation}
\left (\sum_{\alpha\in \mathrm{Irr}(\g)}\frac{n_{\alpha}}{(1+|\alpha|)^{s (\frac{2}{p}-1)}}\norm{\widehat{f}(\alpha)}_{HS}^2 \right )^{\frac{1}{2}}\leq K \norm{f}_{L^p(\g)}~for~all~f\in L^p(\g)
\end{equation}
\item  $s\geq \gamma$.
\end{enumerate}
\end{corollary}

\begin{example}
For connected compact Lie groups $G$ and the Poisson semigroup $(e^{-t(-\Delta)^{\frac{1}{2}}})_{t>0}$, the corresponding Sobolev embedding properties are as follows: For any $1<p<q<\infty$ there exists a constant $K=K(p,q)>0$ such that
\begin{equation}
\norm{(1+(-\Delta)^{\frac{1}{2}})^{-n(\frac{1}{p}-\frac{1}{q})}(f)}_{L^q(G)}\leq K\norm{f}_{L^p(G)}~for~all~f\in L^p(G)
\end{equation}
where $n$ is the real dimension of $G$ and $\Delta$ is the Laplacian operator.
\end{example}

\subsection{Under the growth rate of spheres}

In this section, we provide one more sufficient condition for which inequalities (\ref{ineq-poly2}) are sharp without the existence of a standard noncommutative semigroup. The main ingredient is additional information on lower bounds of the growth rate of $k$-spheres. Indeed, for a compact matrix quantum group $\g$, Corollary \ref{cor-poly} tells us that inequalities (\ref{ineq-poly2}) are sharp if its discrete dual $\widehat{\g}$ satisfies 
\begin{equation}
b_k\lesssim (1+k)^{\gamma}~\mathrm{and~}s_k\gtrsim (1+k)^{\gamma-1}.
\end{equation}

Let us begin with looking at $L^2\rightarrow L^4$ case, which is the dual of $L^{\frac{4}{3}}\rightarrow L^2$ case. The following theorem is motivated by the proof of \cite[Theorem 4.5.2]{Yo18c}, which is for duals of polynomially growing discrete groups.
\begin{theorem}\label{thm-poly}
Let $\g$ be a compact matrix quantum group of Kac type whose dual satisfies
\begin{equation}
b_k\lesssim (1+k)^{\gamma_1}~\mathrm{and~} s_k\gtrsim (1+k)^{\gamma_2-1}~ with~\gamma_1,\gamma_2\geq 1.
\end{equation}
Also, suppose that there exists a constant $K>0$ such that
\begin{equation}
\left \| \sum_{\alpha\in \mathrm{Irr}(\g)}w(\left | \alpha \right |)n_{\alpha} \mathrm{tr}(\widehat{f}(\alpha)u^{\alpha})\right\|_{L^4(\g)}\leq K \norm{f}_{L^2(\g)}~for~all~f\in L^2(\g)
\end{equation}
where $w:\left \{0\right\}\cup \n\rightarrow (0,\infty)$ is a decreasing function. Then we have
\begin{equation}
\limsup_{m\rightarrow \infty}\left \{(1+m)^{\frac{3\gamma_2-2\gamma_1}{4}}  w(m)\right\} <\infty.
\end{equation}
\end{theorem}

\begin{proof}
For a given $f\in L^2(\g)$, we denote by $T_w(f)\sim \displaystyle \sum_{\alpha\in \mathrm{Irr}(\g)}w(\left | \alpha \right |)n_{\alpha} \mathrm{tr}(\widehat{f}(\alpha)u^{\alpha})$ and take $\xi_m=\displaystyle \frac{1}{\sqrt{b_m}}\sum_{\alpha\in B_m}n_{\alpha}\chi_{\alpha}$. Then 
\begin{align*}
1=\norm{\xi_m}_{L^2(\g)}\norm{\xi_m}_{L^2(\g)}&\gtrsim \norm{T_w(\xi_m)}_{L^4(\g)}\norm{T_w(\xi_m)}_{L^4(\g)}\\
&\geq \norm{T_w(\xi_m)T_w(\xi_m)}_{L^2(\g)}
\end{align*}
and 
\begin{align*}
\norm{T_w(\xi_m)T_w(\xi_m)}_{L^2(\g)}^2&=\norm{\sum_{\alpha_1,\alpha_2\in B_m} \frac{1}{b_m} w(\left |\alpha_1 \right |)w(\left |\alpha_2 \right |)n_{\alpha_1}n_{\alpha_2}\chi_{\alpha_1}\chi_{\alpha_2}}_{L^2(\g)}^2\\
&=\frac{1}{b_m^2}\sum_{\sigma\in \mathrm{Irr}(\g)}  \left | h(\sum_{\alpha_1,\alpha_2\in B_m}w(\left |\alpha_1 \right |)w(\left |\alpha_2 \right |)n_{\alpha_1}n_{\alpha_2}\chi_{\alpha_1}\chi_{\alpha_2} \chi_{\sigma}^*) \right |^2\\
&=\frac{1}{b_m^2}\sum_{\sigma\in \mathrm{Irr}(\g)}\left | \sum_{\alpha_1,\alpha_2\in B_m} w(\left | \alpha_1 \right | )w(\left | \alpha_2\right |)n_{\alpha_1}n_{\alpha_2}\delta_{\overline{\alpha_2}\subseteq \overline{\sigma}\tiny \tp \alpha_1}  \right |^2.
\end{align*}
Then since the sequence $(w(k))_{k\geq 0}$ is decreasing, 
\begin{align*}
&\geq \frac{w(m)^4}{b_m^2} \sum_{k=0}^m \sum_{\sigma: \left |\sigma \right |=k}  \left | \sum_{\alpha_1,\alpha_2\in B_m} n_{\alpha_1}n_{\alpha_2}\delta_{\overline{\alpha_2}\subseteq \overline{\sigma} \tiny \tp \alpha_1 } \right |^2\\
&\geq \frac{w(m)^4}{b_m^2} \sum_{k=0}^m \sum_{\sigma: \left |\sigma \right |=k}  \left | \sum_{\alpha_1 \in B_{m-k}} n_{\alpha_1} \sum_{\overline{\alpha_2} \subseteq \overline{\sigma} \tiny \tp \alpha_1 }n_{\alpha_2} \right |^2\\
& =  \frac{w(m)^4}{b_m^2}\sum_{k=0}^m \sum_{\sigma: \left |\sigma\right |=k} \left |\sum_{\alpha_1\in B_{m-k}}n_{\alpha_1}^2 n_{\sigma} \right |^2\\
&= \frac{w(m)^4}{b_m^2} \sum_{k=0}^m \sum_{\sigma: \left |\sigma\right |=k} b_{m-k}^2 n_{\sigma}^2\\
&= \frac{w(m)^4}{b_m^2} \sum_{k=0}^m  b_{m-k}^2 s_k \gtrsim \frac{w(m)^4}{b_m^2} \sum_{k=0}^m  (m-k)^{2\gamma_2}k^{\gamma_2-1}
\end{align*}
for any $m\in \n$. Hence, we obtain
\begin{align*}
\liminf_{m\rightarrow \infty} \frac{b_m^2}{m^{3\gamma_2}w(m)^4}& \gtrsim \lim_{m\rightarrow \infty}\sum_{k=1}^m (1-\frac{k}{m})^{2\gamma_2}(\frac{k}{m})^{\gamma_2-1}\frac{1}{m}\\
&=\int_{0}^1 (1-x)^{2\gamma_2}x^{\gamma_2-1}dx>0,
\end{align*}
so that $\displaystyle \limsup_{m\rightarrow \infty} m^{3\gamma_2-2\gamma_1 } w(m)^4 <\infty$.

\end{proof}

\begin{corollary}\label{cor-poly}
Let $\g$ be a compact matrix quantum group of Kac type whose dual satisfies
\begin{equation}
b_k\lesssim (1+k)^{\gamma}~\mathrm{and~} s_k\gtrsim (1+k)^{\gamma-1}.
\end{equation}
Then the following are equivalent:
\begin{enumerate}
\item For any $1<p\leq 2$, we have
\begin{equation}
\left (\sum_{\alpha\in \mathrm{Irr}(\g)}\frac{n_{\alpha}}{(1+|\alpha|)^{s(\frac{2}{p}-1)}}\norm{\widehat{f}(\alpha)}_{HS}^2 \right )^{\frac{1}{2}}\lesssim \norm{f}_{L^p(\g)}~for~all~f\in L^p(\g).
\end{equation}
\item For some $1<p\leq 2$, we have
\begin{equation}
\left (\sum_{\alpha\in \mathrm{Irr}(\g)}\frac{n_{\alpha}}{(1+|\alpha|)^{s(\frac{2}{p}-1)}}\norm{\widehat{f}(\alpha)}_{HS}^2\right )^{\frac{1}{2}}\lesssim \norm{f}_{L^p(\g)}~for~all~f\in L^p(\g).
\end{equation}
\item $s\geq \gamma$.
\end{enumerate}
\end{corollary}
\begin{proof}
(3) $\Rightarrow$ (1) $\Rightarrow$ (2) is clear from the inequality (\ref{ineq-poly2}). Let us prove (2) $\Rightarrow$ (3). In the case $1<p\leq \displaystyle \frac{4}{3}$, then the Fourier transform $\mathcal{F}:f\mapsto \widehat{f}$ satisfies 
\begin{equation}
\norm{\mathcal{F}}_{L^p(\g)\rightarrow \ell^2( \left \{H_{\alpha}\right\}_{\alpha\in \mathrm{Irr}(\g)},\mu)},\norm{\mathcal{F}}_{L^2(\g)\rightarrow \ell^2(\widehat{\g})}<\infty
\end{equation}
where $\displaystyle \mu(\alpha)=  (1+|\alpha|)^{-s(\frac{2}{p}-1)}$. Then, due to equality (2) in \cite{RiXu11} or (\ref{CI2}), we have
\begin{equation}
(\ell^2( \left \{H_{\alpha}\right\}_{\alpha\in \mathrm{Irr}(\g)},\mu),\ell^2(\widehat{\g}))_{1-\theta}=\ell^2( \left \{H_{\alpha}\right\}_{\alpha\in \mathrm{Irr}(\g)},\mu^{\theta})~\mathrm{for~all~}0<\theta<1.
\end{equation}

Therefore, we have $\displaystyle \norm{\mathcal{F} }_{L^{\frac{4}{3}}(\g)\rightarrow \ell^2( \left \{H_{\alpha}\right\}_{\alpha\in \mathrm{Irr}(\g)},\mu^{\frac{p}{2(2-p)}})}<\infty$ at $\theta=\displaystyle \frac{p}{2(2-p)}$ $(\Leftrightarrow \displaystyle \frac{\theta}{p}+\frac{1-\theta}{2}=\frac{3}{4})$, i.e.
\begin{equation}
\left ( \sum_{\alpha\in \mathrm{Irr}(\g)} \frac{n_{\alpha}}{(1+|\alpha|)^{\frac{s}{2}}} \norm{\widehat{f}(\alpha)}_{HS}^2 \right )^{\frac{1}{2}} \lesssim \norm{f}_{L^{\frac{4}{3}}(\g)}.
\end{equation}

Then the dual statement is
\begin{equation}
\norm{\sum_{\alpha\in \mathrm{Irr}(\g)}\frac{n_{\alpha}}{(1+|\alpha|)^{\frac{s}{4}}}\mathrm{tr}(\widehat{f}(\alpha)u^{\alpha})}_{L^4(\g)} \lesssim \norm{f}_{L^{2}(\g)}
\end{equation}
and we obtain $\displaystyle s\geq \gamma$ from Theorem \ref{thm-poly}.

In the case $\displaystyle \frac{4}{3}<p<2$, let us choose $\displaystyle p_0\in (1,\frac{4}{3})$. Then by the inequality (\ref{ineq-poly2}) and the given assumption, we have
\begin{equation}
\norm{\mathcal{F}}_{L^{p_0}(\g)\rightarrow \ell^2( \left \{H_{\alpha}\right\}_{\alpha\in \mathrm{Irr}(\g)},\mu^r)},\norm{\mathcal{F}}_{L^p(\g)\rightarrow \ell^2 (\left \{H_{\alpha}\right\}_{\alpha\in \mathrm{Irr}(\g)},\mu)}<\infty ,
\end{equation}
where $r=\displaystyle \frac{\gamma(\frac{2}{p_0}-1)}{s(\frac{2}{p}-1)}$. Then $\norm{\mathcal{F}}_{L^{\frac{4}{3}}(\g)\rightarrow \ell^2 (\left \{H_{\alpha}\right\}_{\alpha\in \mathrm{Irr}(\g)},\mu^{r\theta + (1-\theta)}) }<\infty$ for $\theta\in (0,1)$ satisfying $\displaystyle \frac{\theta}{p_0}+\frac{1-\theta}{p}=\frac{3}{4}$ by (\ref{CI2}) and the dual argument shows
\begin{equation}
\norm{\sum_{\alpha\in \mathrm{Irr}(\g)}\frac{n_{\alpha}}{(1+|\alpha|)^{s(\frac{1}{p}-\frac{1}{2})\cdot (r\theta+(1-\theta))}}\mathrm{tr}(\widehat{f}(\alpha)u^{\alpha})}_{L^4(\g)} \lesssim \norm{f}_{L^{2}(\g)}.
\end{equation}

Then it is easy to check that 
\begin{align*}
s(\frac{1}{p}-\frac{1}{2})\cdot (r\theta+1-\theta)&=\gamma(\frac{1}{p_0}-\frac{1}{2})\theta+s(\frac{1}{p}-\frac{1}{2})\cdot (1-\theta)\\
&=(s-\gamma)(\frac{1}{p}-\frac{1}{2})(1-\theta)+\gamma\cdot [(\frac{1}{p_0}-\frac{1}{2})\theta  + (\frac{1}{p}-\frac{1}{2})(1-\theta) ]\\
&=(s-\gamma)(\frac{1}{p}-\frac{1}{2})(1-\theta)+\gamma(\frac{3}{4}-\frac{1}{2})\\
&=(s-\gamma)(\frac{1}{p_0}-\frac{1}{2})(1-\theta)+\frac{\gamma}{4}
\end{align*}
and Theorem \ref{thm-poly} asserts that $\displaystyle (s-\gamma)(\frac{1}{p_0}-\frac{1}{2})(1-\theta)+\frac{\gamma}{4}\geq \frac{\gamma}{4}$. Hence, we can conclude that $s\geq \gamma$.
\end{proof}

\begin{example}
\begin{enumerate}
\item For any $1<p\leq 2$ sharp Sobolev embedding properties on duals of polynomially growing discrete groups are
\begin{equation}
\left ( \sum_{g\in \Gamma} \frac{|f(g)|^2}{(1+k)^{\gamma(\frac{2}{p}-1)}}\right )^{\frac{1}{2}}\lesssim \norm{\lambda(f)}_{L^p(\widehat{\Gamma})},
\end{equation}
where $\lambda(f)\sim \displaystyle \sum_{g\in \Gamma}f(g)\lambda_g\in L^p(\widehat{\Gamma})$ and $\gamma$ is the polynomial growth order of $\Gamma$.  For growth rates of $k$-spheres, see \cite[Corollary 11]{BrLe13}.
\item For any $1<p\leq 2$ sharp Sobolev embedding properties on $O_2^+$ or $S_4^+$ are
\begin{equation}
\left ( \sum_{k\geq 0} \frac{n_k}{(1+k)^{3(\frac{2}{p}-1)}}\norm{\widehat{f}(k)}_{HS}^2\right )^{\frac{1}{2}}\lesssim \norm{f}_{L^p(\g)}.
\end{equation}
\end{enumerate}
\end{example}

\section{Sobolev embedding properties under the rapid decay property}\label{sec-rd1}

Arguably, the most important examples of compact quantum groups are duals of free groups $\widehat{\mathbb{F}_N}$ and free quantum groups such as $O_{N+1}^+,S_{N+3}^+$, Since their discrete duals are exponentially growing if $N\geq 2$, the results from Section \ref{sec-poly} are not applicable.

The compact quantum groups $\widehat{\mathbb{F}_N},O_{N+1}^+,S_{N+3}^+$ with $N\geq 2$ have unique natures in view of analysis. Those are not co-amenable, the underlying $C^*$-algebras are non-nuclear, etc. One notable known result \cite[Theorem 3.2]{Yo18b} is that Hausdorff-Young inequalities can be improved in the case of free orthogonal quantum groups $O_N^+$ with $N\geq 3$.

The aim of this Section is
\begin{itemize}
\item (Theorem \ref{thm-rapid}) to generalize the proof of \cite[Theorem 3.2]{Yo18b} to general compact quantum groups whose duals have the rapid decay property and
\item (Theorem \ref{thm-main-rapid}) to combine Theorem \ref{thm-main-rapid} and \cite[Corollary 3.9]{Yo18a} to establish Sobolev embedding properties under the rapid decay property.
\end{itemize}

\subsection{Sharpened Hausdorff-Young inequalities}\label{sec-HY}

For general compact quantum groups, the Hausdorff-Young inequalities state that the Fourier transform $\mathcal{F}:L^p(\g)\rightarrow \ell^{p'}(\widehat{\g})$ is contractive for all $1\leq p\leq 2$. Moreover, boundedness of a multiplier
\begin{equation}
\mathcal{F}_w:L^p(\g)\rightarrow \ell^{p'}(\widehat{\g}),~f\mapsto (w(\alpha)\widehat{f}(\alpha))_{\alpha \in \mathrm{Irr}(\g)},
\end{equation}
implies boundedness of the sequence $w=(w(\alpha))_{\alpha \in \mathrm{Irr}(\g)}$ if $\g$ is one of the compact quantum groups listed below:
\begin{itemize}
\item connected semisimple compact Lie groups
\item duals of discrete groups $\Gamma$
\item free orthogonal quantum group $O_2^+$
\item quantum $SU(2)$ group
\end{itemize}

The above observation is motivated by \cite[Section 4]{Yo18b} and boundedness of $w$ can be explained by \cite[Main theorem]{GiTr80}, families of matrix elements $\left \{\lambda_g\right\}_{g\in \Gamma}$, $\left \{u^n_{0,0} \right\}_{n\geq 0}$ and $\left \{ u^n_{n,n} \right\}_{n\geq 0}$ respectively with respect to canonical choices of orthonormal bases.

Nevertheless, \cite[Theorem 3.2]{Yo18b} has established the existence of an unbounded (exponentially) increasing sequence for the case of $O_N^+$ with $N\geq 3$, and we will adapt the proof of \cite[Theorem 3.2]{Yo18b} to general compact matrix quantum groups under the rapid decay property. 

\begin{theorem}[A sharpened Hausdorff-Young inequality]\label{thm-rapid}
Let $\g$ be a compact matrix quantum group of Kac type whose dual $\widehat{\g}$ has the rapid decay property with $r_k\lesssim (1+k)^{\beta}$. Then for any $1<p\leq 2$ we have
\begin{equation}
\left (\sum_{k\geq 0}\frac{1}{(1+k)^{\beta(p'-2)}}(\sum_{\alpha\in S_k}n_{\alpha}\norm{\widehat{f}(\alpha)}_{S^2_{n_{\alpha}}}^2)^{\frac{p'}{2}} \right )^{\frac{1}{p'}}\lesssim \norm{f}_{L^p(\g)}~for~all~f\in L^p(\g).
\end{equation}
\end{theorem}

\begin{proof}
First of all, \cite[Proposition 3.7]{Yo18a} states that a linear map
\begin{equation}
\Phi:L^1(\g)\rightarrow \ell^{\infty}-(\left \{ H_k \right\}_{k\geq 0} )  , f\mapsto \bigoplus_{k\geq 0}\left (\frac{1}{(1+k)^{\beta}}\widehat{f}(\alpha)\right )_{\alpha\in S_k}
\end{equation}
is bounded and the Plancherel identity can be interpreted as that
\begin{equation}
\Phi:L^2(\g)\rightarrow \ell^2-(\left \{H_k\right\}_{k\geq 0},\mu)
\end{equation}
is an isometry where $\mu(k)=(1+k)^{2\beta}$. Since 
\begin{equation}
(\ell^{\infty}-(\left \{ H_k \right\}_{k\geq 0} ),\ell^2-(\left \{H_k\right\}_{k\geq 0},\mu))_{\theta}=\ell^{q}-(\left \{H_k\right\}_{k\geq 0},\mu) 
\end{equation}
where $\displaystyle \frac{1-\theta}{\infty}+\frac{\theta}{2}=\frac{1}{q}$ with $0<\theta<1$ \cite{Xu96}, the resulting inequality at $\theta=\displaystyle \frac{2}{p'}$ is
\begin{equation}
\left (\sum_{k\geq 0}\frac{1}{(1+k)^{\beta(p'-2)}}(\sum_{\alpha\in S_k}n_{\alpha}\norm{\widehat{f}(\alpha)}_{S^2_{n_{\alpha}}}^2)^{\frac{p'}{2}}\right )^{\frac{1}{p'}}\lesssim \norm{f}_{L^p(\g)}
\end{equation}
for all $1< p\leq 2$.

\end{proof}

\begin{remark}
The above Theorem \ref{thm-rapid} is sharp for $\g=\widehat{\mathbb{F}_N},O_{N+1}^+$ or $S_{N+3}^+$ with $N\geq 2$ by Corollary \ref{cor-application}.
\end{remark}

\begin{corollary}\label{cor-rapid}
Under the same assuption of Theorem \ref{thm-rapid}, for any $1<p\leq 2$, we have
\begin{equation}
\left (\sum_{\alpha\in \mathrm{Irr}(\g)}\frac{n_{\alpha}^{\frac{p'}{2}-1}}{(1+\left |\alpha \right |)^{\beta(p'-2)}}\cdot n_{\alpha}\norm{\widehat{f}(\alpha)}_{S^{p'}_{n_{\alpha}}}^{p'}\right )^{\frac{1}{p'}}\lesssim \norm{f}_{L^p(\g)}~for~all~f\in L^p(\g)
\end{equation}

\end{corollary}
\begin{proof}
It is enough to note that $\displaystyle \sum_{\alpha\in S_k}n_{\alpha}^{\frac{p'}{2}}\norm{A(\alpha)}_{S^2_{n_{\alpha}}}^{p'}\leq ( \sum_{\alpha\in S_k}n_{\alpha}\norm{A(\alpha)}_{S^2_{n_{\alpha}}}^{2} )^{\frac{p'}{2}}$ and $\norm{A}_{S^{p'}_n}\leq \norm{A}_{S^2_n}$.
\end{proof}

\begin{remark}
In view of Corollary \ref{cor-rapid}, if $\displaystyle \sup_{\alpha\in \mathrm{Irr}(\g)}\frac{n_{\alpha}}{(1+\left |\alpha \right |)^{2\beta}}=\infty$, we are able to find an unbounded sequence $w=(w(\alpha))_{\alpha\in \mathrm{Irr}(\g)}$ such that
\begin{equation}
L^p(\g)\rightarrow \ell^{p'}(\widehat{\g}),~f\mapsto (w(\alpha)\widehat{f}(\alpha))_{\alpha\in \mathrm{Irr}(\g)},~ is~bounded.
\end{equation}
This happens when $\g$ is one of the following:
\begin{itemize}
\item Free orthogonal quantum groups $O_N^+$ with $N\geq 3$;
\item Free unitary quantum groups $U_N^+$ with $N\geq 3$;
\item Quantum automorphism group $\g_{aut}(B,\psi)$ with a $\delta$-trace $\psi$ and $\mathrm{dim}(B)\geq 5$.
\end{itemize}
\end{remark}

\subsection{Sobolev embedding properties under the rapid decay property}

In this section, we will present Sobolev embedding properties under the rapid decay property by interpolating Theorem \ref{thm-rapid} and Hardy-Littlewood inequalities \cite[Theorem 3.8]{Yo18a}.

\begin{theorem}\label{thm-main-rapid}
Let $\g$ be a compact matrix quantum group of Kac type whose dual $\widehat{\g}$ has the rapid decay property with $r_k\lesssim (1+k)^{\beta}$. Then for any $1<p\leq 2$ we have
\begin{equation}
\left ( \sum_{\alpha\in \mathrm{Irr}(\g)} \frac{n_{\alpha}}{(1+|\alpha|)^{(2\beta+1) (\frac{2}{p}- 1 )}} \norm{\widehat{f}(\alpha)}_{HS}^2 \right )^{\frac{1}{2}}\lesssim \norm{f}_{L^p(\g)}~for~all~f\in L^p(\g).
\end{equation}
\end{theorem}

\begin{proof}
By \cite[Corollary 3.9]{Yo18a} and Theorem \ref{thm-rapid}, the Fourier transform $\mathcal{F}:f\mapsto \widehat{f}$ satisfies
\begin{equation}
\norm{\mathcal{F}}_{L^p(\g)\rightarrow \ell^p-(\left \{H_k\right\}_{k\geq 0},\mu_0)},\norm{\mathcal{F}}_{L^p(\g)\rightarrow \ell^{p'}-(\left \{H_k\right\}_{k\geq 0},\mu_1)}<\infty
\end{equation}
where $\mu_0(k)=(1+k)^{-(\beta+1)(2-p)}$ and $\mu_1(k)=(1+k)^{-\beta(p'-2)}$. Hence, $\displaystyle (\mu_0^{\frac{1}{p}}\mu_1^{\frac{1}{p'}})(k)=(1+k)^{-(2\beta+1)(\frac{2}{p}-1)}$, so that we reach a conclusion.
\end{proof}

\begin{corollary}
Let $s\geq 2\beta +1$ and $\g$ be a compact matrix quantum group of Kac type whose dual has the rapid decay property with $r_k\lesssim (1+k)^{\beta}$. Suppose that there exists a standard non-commutative semigroup $(T_t)_{t>0}$ whose infinitesimal generator $L$ satisfies
\begin{equation}
L(u^{\alpha}_{i,j})=-l(\alpha)u^{\alpha}_{i,j}~and~l(\alpha)\sim |\alpha|.
\end{equation}
Then for any $1<p<q<\infty$ we have
\begin{equation}
\norm{(1-L)^{-s(\frac{1}{p}-\frac{1}{q})}(f)}_{L^q(\g)}\lesssim \norm{f}_{L^p(\g)}~for~all~f\in L^p(\g).
\end{equation}

\end{corollary}

\begin{proof}
It is enough to note that
\begin{equation}\label{ineq-Sobolev-Poisson}
\left ( \sum_{\alpha\in \mathrm{Irr}(\g)} \frac{n_{\alpha}}{(1+|\alpha|)^{s(\frac{2}{p}- 1 )}} \norm{\widehat{f}(\alpha)}_{HS}^2 \right )^{\frac{1}{2}}\approx \norm{(1-L)^{-s(\frac{1}{p}-\frac{1}{2})}(f)}_{L^2(\g)}.
\end{equation}
Then Theorem \ref{thm-main-rapid} and Theorem \ref{thm-Xiao} complete the proof.

\end{proof}

\section{Rapid decay degree of discrete quantum groups}\label{sec-rd2}

The rapid decay degree, which was suggested in \cite{Ni10}, is a way to quantify the rapid decay property of a discrete group, and the notion naturally extends to the framework of duals of compact matrix quantum groups of Kac type. A natural way is to define the degree of rapid decay property $\mathrm{rd}(\widehat{\g})$ as the infimum of positive numbers $s\geq 0$ satisfying
\begin{equation}
\norm{f}_{L^{\infty}(\g)}\lesssim \left ( \sum_{\alpha\in \mathrm{Irr}(\g)}(1+|\alpha|)^{2s}n_{\alpha}\norm{\widehat{f}(\alpha)}_{HS}^2 \right )^{\frac{1}{2}}~\mathrm{for~all~}f\in \mathrm{Pol}(\g).
\end{equation}
Note that this definition is independent of the choice of a generating unitary representation.

For discrete groups, it has turned out that $\mathrm{rd}(\Gamma)=\displaystyle \frac{\gamma}{2}$ for any finitely generated discrete group $\Gamma$ with the polynomial growth order $\gamma$ \cite{Ni10} and that $\mathrm{rd}(\Gamma)=\displaystyle \frac{3}{2}$ for any non-elementary hyperbolic groups \cite{Ni17}. 

As quantum analogues of the above results, we aim to compute the rapid decay degree of polynomially growing discrete quantum groups and duals of free quantum groups $\widehat{O_N^+},\widehat{S_N^+}$. By theorem \ref{thm-poly-ultra}, we are already ready to extend \cite[Proposition 2.2 (2)]{Ni10}.

\begin{proposition}
Let $\g$ be a compact matrix quantum group of Kac type whose dual satisfies $b_k\approx (1+k)^{\gamma}$. Then
\begin{equation}
\norm{\sum_{\alpha\in \mathrm{Irr}(\g)}\frac{n_{\alpha}}{(1+|\alpha|)^s}\mathrm{tr}(\widehat{f}(\alpha)u^{\alpha})}_{L^{\infty}(\g)}\lesssim \norm{f}_{L^2(\g)}~for~all~f\in L^2(\g)
\end{equation}
if and only if $\displaystyle s>\frac{\gamma}{2}$. In particular, $\mathrm{rd}(\widehat{\g})=\displaystyle \frac{\gamma}{2}$.
\end{proposition}

\begin{proof}
First of all, we take a positive function $w(\alpha)=\displaystyle \log ( (1+|\alpha|)^s )$ with $\displaystyle  s > \frac{\gamma}{2}$ to apply Theorem \ref{thm-poly-ultra} (1). Then
\begin{align*}
C_w&=\sum_{\alpha \in \mathrm{Irr}(\g)}\frac{n_{\alpha}^2}{(1+|\alpha|)^{2s}}\\
&=\lim_{n\rightarrow \infty}\frac{b_n}{(1+n)^{2s}}+\sum_{k=0}^{\infty}b_k  \left (\frac{1}{(1+k)^{2s}}-\frac{1}{(2+k)^{2s}}\right )  \\
&\lesssim \sum_{k=0}^{\infty} (1+k)^{\gamma-(2s+1)} <\infty.
\end{align*}

On the other hand, if we assume the following inequality
\begin{equation}
\norm{\sum_{\alpha\in \mathrm{Irr}(\g)}\frac{n_{\alpha}}{(1+|\alpha|)^s}\mathrm{tr}(\widehat{f}(\alpha)u^{\alpha})}_{L^{\infty}(\g)}\lesssim \norm{f}_{L^2(\g)},
\end{equation}
then Theorem \ref{thm-poly-ultra} (2) implies $\displaystyle  \sum_{k=0}^{\infty}(2+k)^{\gamma-(2s+1)}\lesssim \sum_{\alpha\in \mathrm{Irr}(\g)}\frac{n_{\alpha}^2}{(1+|\alpha|)^{2s}}<\infty$ as in the above. Therefore, $\displaystyle s>\frac{\gamma}{2}$.
\end{proof}

In the case that $\widehat{\g}$ is exponentially growing, the following approach is valid under the rapid decay property. The proof relies on standard arguments that have been already used in \cite{Ni10,Br12,JuPaPaPe17,FrHoLeUlZh17}.

\begin{proposition}\label{prop-rd}
Let $\g$ be a compact matrix quantum group of Kac type whose dual has the rapid decay property with $r_k\lesssim (1+k)^{\beta}$ and let $w:\left \{0\right\}\cup \n\rightarrow (0,\infty)$ be a positive function such that $\displaystyle C_w=\sum_{k\geq 0}\frac{(1+k)^{2\beta}}{e^{2w(k)}}<\infty$. Then we have
\begin{equation}
\norm{\sum_{\alpha\in \mathrm{Irr}(\g)}\frac{n_{\alpha}}{e^{w(|\alpha|)}}\mathrm{tr}(\widehat{f}(\alpha)u^{\alpha})}_{L^{\infty}(\g)} \leq \sqrt{C_w} \norm{f}_{L^2(\g)}~for~all~f\in L^2(\g).
\end{equation}

In particular, for any $s>\displaystyle \beta+\frac{1}{2}$ we have
\begin{equation}
\norm{\sum_{\alpha\in \mathrm{Irr}(\g)}\frac{n_{\alpha}}{(1+|\alpha|)^{s}}\mathrm{tr}(\widehat{f}(\alpha)u^{\alpha})}_{L^{\infty}(\g)} \lesssim \norm{f}_{L^2(\g)}~for~all~f\in L^2(\g).
\end{equation}
\end{proposition}
\begin{proof}
It is enough to see that for any $f\in \mathrm{Pol}(\g)$
\begin{align*}
\norm{\sum_{\alpha\in \mathrm{Irr}(\g)}\frac{n_{\alpha}}{e^{w(|\alpha|)}}\mathrm{tr}(\widehat{f}(\alpha)u^{\alpha})}_{L^{\infty}(\g)}&\leq \sum_{k\geq 0}\frac{\norm{p_k(f)}_{L^{\infty}(\g)}}{e^{w(|\alpha|)}}\\
&\lesssim \sum_{k\geq 0}\frac{(1+k)^{\beta}\norm{p_k(f)}_{L^2(\g)}}{e^{w(|\alpha|)}}\\
&\leq \left ( \sum_{k\geq 0}\frac{(1+k)^{2\beta}}{e^{2w(|\alpha |)}} \right )^{\frac{1}{2}}\norm{f}_{L^2(\g)}.
\end{align*}

\end{proof}

From now on, let us try to detect the rapid decay degree of duals of free quantum groups.

\begin{theorem}\label{thm-rd}
Let $\g$ be a compact matrix quantum group of Kac type and $w:\left \{0\right\}\cup \n \rightarrow (0,\infty)$ be a positive function. If we suppose that
\begin{equation}
\norm{\sum_{\alpha\in \mathrm{Irr}(\g)}\frac{n_{\alpha}}{e^{w(|\alpha |)}}\mathrm{tr}(\widehat{f}(\alpha)u^{\alpha})}_{L^{\infty}(\g)}\leq C \norm{f}_{L^2(\g)}~for~all~f\in L^2(\g),
\end{equation}
then there exists a universal constant $K>0$ such that $\displaystyle \sum_{k\geq 0}\frac{(1+k)^2}{e^{2w(k)}}\leq KC^2$ if $\g$ is one of the following:
\begin{itemize}
\item duals of non-elementary hyperbolic broups;
\item free orthogonal quantum groups $O_N^+$ with $N\geq 2$;
\item free permutation quantum groups $S_N^+$ with $N\geq 4$.
\end{itemize}
\end{theorem}

\begin{proof}
First of all, let $\Gamma$ be a non-elementary hyperbolic group and $\displaystyle \sigma_k=\sum_{g\in S_k} \lambda_g$. Then for any positive sequence $(a_k)_{k \geq 0}$ the main theorem in \cite{Ni17} states that 
\begin{equation}
\norm{\sum_{k\geq 0} \frac{a_k}{\sqrt{s_k}}\sigma_k}_{L^{\infty}(\widehat{\Gamma})}\approx \sum_{k\geq 0}(k+1)a_k.
\end{equation}
Therefore, from the given assumption, we have
\begin{align*}
C(\sum_{k\geq 0}a_k^2)^{\frac{1}{2}}&=C\norm{\sum_{k\geq 0}\frac{a_k}{\sqrt{s_k}}\sigma_k}_{L^2(\widehat{\Gamma})}\\
&\geq \norm{\sum_{k\geq 0}\frac{a_ke^{-w(k)}}{\sqrt{s_k}}\sigma_k}_{L^{\infty}(\widehat{\Gamma})}\approx \sum_{k\geq 0}a_k (k+1)e^{-w(k)}.
\end{align*}
Since $(a_k)_{k\geq 0}$ is arbitrary, we have $\displaystyle \sum_{k\geq 0}(k+1)^2e^{-2w(k)}\lesssim C^2$.

Secondly, let $\g$ be $O_N^+$ (resp. $S_{N+2}^+$) with $N\geq 2$. Then for the associated compact Lie group $SU(2)$ (resp. $SO(3)$), \cite[Lemma 4.7]{Yo18a} states that for any $1\leq p\leq \infty$ and any $f\sim \displaystyle \sum_{k\geq 0}a_k\chi_k$, we have $\norm{f}_{L^p(\g)}=\norm{\widetilde{f}}_{L^p(SU(2))~(\mathrm{resp.~}L^p(SO(3)))}$ where $\widetilde{f}\sim\displaystyle \sum_{k\geq 0}a_k\widetilde{\chi_k}$ and $\widetilde{\chi_k}$ is the $k$-th character of $SU(2)$ (resp. $SO(3)$). Also, note that there exists the Poisson semigroup $(\mu_t)_{t>0}\subseteq L^1(SU(2))$ (resp. $L^1(SO(3))$) satisfying $\mu_t\sim \displaystyle \sum_{k\geq 0} e^{-t\kappa_{k}^{\frac{1}{2}}} (k+1) \widetilde{\chi_k}$ (resp. $\mu_t\sim \displaystyle \sum_{k\geq 0} e^{-t\kappa_{k}^{\frac{1}{2}}} (2k+1) \widetilde{\chi_k}$).

Now, from the given assumption and \cite[Lemma 6.3 (2)]{Yo18a}, we have
\begin{align*}
\sum_{k\geq 0} e^{-t \kappa_k^{\frac{1}{2}}}e^{-2w(k)}(1+k)^2&\approx \norm{\sum_{k\geq 0}e^{-t \kappa_k^{\frac{1}{2}}}e^{-2w(k)}(1+k)\widetilde{\chi_{k}}}_{L^{\infty}(SU(2))~(\mathrm{resp.}L^{\infty}(SO(3)))}\\
&=\norm{\sum_{k\geq 0}e^{-t \kappa_k^{\frac{1}{2}}}e^{-2w(k)}(1+k) \chi_{k}}_{L^{\infty}(\g)}\\
&\leq C \norm{\sum_{k\geq 0}e^{-t \kappa_k^{\frac{1}{2}}}e^{-w(k)}(1+k) \chi_{k}}_{L^2(\g)}\\
&\leq C^2\norm{\widetilde{\mu_t}}_{L^1(\g)}=  C^2 \norm{\mu_t}_{L^1(SU(2))~(\mathrm{resp.}L^1(SO(3)))}=C^2
\end{align*}

Lastly, by taking the limit $t\rightarrow 0^+$, we get $\displaystyle \sum_{k\geq 0}e^{-2w(k)}(1+k)^2\leq K C^2$ for a universal constant $K>0$.
\end{proof}

By combining Proposition \ref{prop-rd} and Theorem \ref{thm-rd}, we can compute the rapid decay degrees of $\widehat{O_N^+}$ and $\widehat{S_N^+}$.

\begin{corollary}\label{cor-rd}
Let $\g$ be $O_N^+$ or $S_{N+2}^+$ with $N\geq 2$ and $s\geq 0$. Then
\begin{equation}\label{eq-1}
\norm{\sum_{k\geq 0}\frac{n_{k}}{(1+k)^s}\mathrm{tr}(\widehat{f}(k)u^{k})}_{L^{\infty}(\g)}\lesssim \norm{f}_{L^2(\g)}~for~all~f\in L^2(\g)
\end{equation}
if and only if $s>\displaystyle \frac{3}{2} $. In particular, $\mathrm{rd}(\widehat{\g})=\displaystyle \frac{3}{2}$.
\end{corollary}

\begin{proof}
It is sufficient to see the only if part. By Theorem \ref{thm-rd}, the given assumption implies $\displaystyle \sum_{k\geq 0} \frac{1}{(1+k)^{2s-2}}<\infty$, so that $s>\displaystyle \frac{3}{2}$.
\end{proof}

\section{Sharp Sobolev embedding properties for $\widehat{\mathbb{F}_N},O_N^+,S_N^+$}\label{sec-Sobolev}

Throughout this section, we will present sharp Sobolev embedding properties for duals of free groups and free quantum groups $O_N^+,S_N^+$. In each case, the standard noncommutative semigroups $T_t^F,T_t^O,T_t^S$ introduced in Example \ref{ex-1} will replace the role of Poisson or heat semigroups in each case.

We begin with the following computational lemma.

\begin{lemma}\label{lem1}
\begin{enumerate}
\item For any $t>0$, 
\begin{equation}\label{eq-2}
\sum_{k\geq 0}\frac{(1+k)^2}{e^{2t(1+k)}} =e^{-2t}(1-e^{-2t})^{-3}(1+e^{-2t}).
\end{equation}
\item If $a_k\sim k$, then $\displaystyle \sup_{0<t<\infty}\left \{ t^s \sum_{k\geq 0}\frac{(1+k)^2}{e^{2t(1+a_k)}} \right\}<\infty$ holds if and only if $s\geq 3$.
\end{enumerate}
\end{lemma}

\begin{proof}
(2) Set $f(t)=\displaystyle \sum_{k\geq 0}\frac{(1+k)^2}{e^{2t(1+k)}}$ and $\displaystyle g(t)=\sum_{k\geq 0}\frac{(1+k)^2}{e^{2t(1+a_k)}}$. Since there exist $D_1,D_2>0$ such that $D_1\cdot (1+k)\leq 1+a_k \leq D_2\cdot (1+k)$, we have
\begin{equation}
f(D_2 t) \leq   g(t) \leq f(D_1 t).
\end{equation}
Thus, it is enough to show that $\displaystyle \sup_{0<t<\infty}\left \{t^s f(t)\right\}<\infty$ if and only if $s \geq 3$. Due to the explicit form (\ref{eq-2}) of $f(t)$, we have $\displaystyle \lim_{t\rightarrow 0^+} t^3 f(t)=\frac{1}{4}$ and $\displaystyle \lim_{t\rightarrow \infty} t^x f(t)=0$ for any $x\geq 0$, which give us the conclusion.
\end{proof}

Now, we are ready to compute the optimal order of the ultracontractivity of $S_t=(e^{-t}T_t)_{t>0}$.

\begin{corollary}\label{cor-ultra-sharp}
\begin{enumerate}
\item Let $\g=\widehat{\mathbb{F}_N}$ with $N\geq 2$ and $T_t=T_t^F$. Then there exists a universal constant $K>0$ such that
\begin{equation}
\norm{S_t(\lambda(f))}_{VN(\mathbb{F}_N)}\leq \frac{K \norm{f}_{\ell^2(\mathbb{F}_N)}}{t^{\frac{s}{2}}}~for~all~f\in \ell^2(\mathbb{F}_N)~and~t>0
\end{equation}
if and only if $s\geq 3$.
\item Let $N\geq 3$, $\g=O_N^+$ (resp. $S_{N+2}^+$) and $T_t=T_t^O$ (resp. $T_t^S$). Then there exists a universal constant $K>0$ such that 
\begin{equation}\label{ineq3}
\norm{S_t(f)}_{L^{\infty}(\g)}\leq \frac{K\norm{f}_{L^2(\g)}}{t^{\frac{s}{2}}}~for~all~f\in L^2(\g)~and~t>0
\end{equation}
if and only if $s \geq 3$.
\end{enumerate}
\end{corollary}

\begin{proof}
Recall that, in the case of (2), $T_t:u_{i,j}^k\mapsto e^{-tc_k}u^k_{i,j}$ with $c_k\sim k$.

Now, in all cases, if we suppose $s \geq 3$, then
\[C_w=\sum_{k\geq 0}\frac{(1+k)^2}{e^{2t(1+k)}}~(resp. \sum_{k\geq 0}\frac{(1+k)^2}{e^{2t(1+c_k)}})\lesssim \frac{1}{t^s}\]
by Lemma \ref{lem1} (2), so Proposition \ref{prop-rd} is applicable.

Conversely, from the assumed inequalities, we obtain
\[\sum_{k\geq 0} \frac{(1+k)^2}{e^{2t(1+k)}} ~(resp. \sum_{k\geq 0}\frac{(1+k)^2}{e^{2t(1+c_k)}}) \lesssim \frac{1}{t^s}\]
by Theorem \ref{thm-rd}, so Lemma \ref{lem1} (2) tells us $s\geq 3$.
\end{proof}

Finally, since we have sharp ultracontractivity properties of $S_t=(e^{-t}T_t)_{t>0}$ (Corollary \ref{cor-ultra-sharp}), we reach the following sharp Sobolev embedding properties for $\widehat{\mathbb{F}_N}$, $O_N^+$ and $S_N^+$ by Theorem \ref{thm-Xiao}.

\begin{example} \label{Sobolev-sharp}
\begin{enumerate}
\item Let $N\geq 2$ and $1<p<q<\infty$. Then
\begin{equation}
\norm{\sum_{g\in \mathbb{F}_N }\frac{f(g)}{(1+\left |g\right |)^{3(\frac{1}{p}-\frac{1}{q})}}\lambda_g}_{L^q(\widehat{\mathbb{F}_N})}\lesssim \norm{\lambda(f)}_{L^p(\widehat{\mathbb{F}_N})}
\end{equation}
for all $\lambda(f)\sim \displaystyle \sum_{g\in \mathbb{F}_N }f(g)\lambda_g\in L^p(\widehat{\mathbb{F}_N })$.
\item Let $\g=O_N^+$ or $S_{N+2}^+$ with $N\geq 3$ and $1<p\leq 2$. Then
\begin{equation}
\left ( \sum_{k\geq 0}\frac{n_k}{(1+k)^{3(\frac{2}{p}-1)}}\norm{\widehat{f}(k)}_{HS}^2 \right )^{\frac{1}{2}}\lesssim \norm{f}_{L^p(\g)}~for~all~f\in L^p(\g).
\end{equation}
\end{enumerate}
\end{example}

The above results are applicable to show the following:
\begin{itemize}
\item Hardy-Littlewood inequalities on $\widehat{\mathbb{F}_N }$ \cite[Theorem 4.4]{Yo18a} are sharp;
\item Theorem \ref{thm-rapid} for $\widehat{\mathbb{F}_N }$ and $O_{N+1}^+,S_{N+3}^+$ with $N\geq 2$ are sharp.
\end{itemize}

\begin{corollary}\label{cor-application}
\begin{enumerate}
\item Let $N\geq 2$ and $1<p\leq 2$. Then 
\begin{equation}
\left (\sum_{k\geq 0} \frac{1}{(1+k)^{s}} \left (\sum_{g\in \mathbb{F}_N  :|g|=k}|f(g)|^2 \right )^{\frac{p}{2}} \right )^{\frac{1}{p}}\lesssim \norm{\lambda(f)}_{L^p(\widehat{\mathbb{F}_N })}
\end{equation}
for all $\lambda(f)\sim \displaystyle \sum_{g\in \mathbb{F}_N  }f(g)\lambda_g\in L^p(\widehat{\mathbb{F}_N })$ if and only if $s\geq 4-2p$.
\item Let $1<p\leq 2$ and $\g$ be $\widehat{\mathbb{F}_N},O_{N+1}^+$ or $S_{N+3}^+$ with $N\geq 2$. Then
\begin{equation}
\left ( \sum_{k\geq 0} \frac{1}{(1+k)^s} \left ( \sum_{\alpha\in S_k} n_{\alpha}\norm{\widehat{f}(\alpha)}_{HS}^2 \right )^{\frac{p'}{2}} \right )^{\frac{1}{p'}} \lesssim \norm{f}_{L^p(\g)}
\end{equation}
for all $f\in L^p(\g)$ if and only if $s\geq p'-2$.
\end{enumerate}
\end{corollary}
\begin{proof}
In both cases, it is sufficient to show the only if parts.
\begin{enumerate}
\item From the assumption and Theorem \ref{thm-rapid}, we know that the Fourier transform $\F$ satisfies 
\[\norm{\F}_{L^p(\widehat{\mathbb{F}_N  })\rightarrow \ell^p(\left \{H_k\right\}_{k\geq 0},\mu_0)},~\norm{\F}_{L^p(\widehat{\mathbb{F}_N })\rightarrow \ell^{p'}(\left \{H_k\right\}_{k\geq 0},\mu_1)}<\infty \] 
where $\mu_0(k)=(1+k)^{-s}$ and $\mu_1(k)=(1+k)^{2-p'}$. By (\ref{CI1}), we have 
\[\norm{\F}_{L^p(\widehat{\mathbb{F}_N })\rightarrow \ell^{2}(\left \{H_k\right\}_{k\geq 0},\mu)}<\infty~\mathrm{with}~\mu(k)=(1+k)^{-\frac{s}{p}+\frac{2}{p'}-1}\] 
and Example \ref{Sobolev-sharp} (1) tells us that $\displaystyle \frac{s}{p}-\frac{2}{p'}+1\geq 3(\frac{2}{p}-1)$ ($\Leftrightarrow s\geq 4-2p$).
\item \cite[Corollary 3.9]{Yo18a} and the given assumption tell us that 
\[\norm{\F}_{L^p(\g)\rightarrow \ell^p(\left \{H_k\right\}_{k\geq 0},\mu_0)},~\norm{\F}_{L^p(\g)\rightarrow \ell^{p'}(\left \{H_k\right\}_{k\geq 0},\mu_1)}<\infty \] 
where $\mu_0(k)=(1+k)^{2p-4}$ and $\mu_1(k)=(1+k)^{-s}$. By (\ref{CI1}), we have 
\[\norm{\F}_{L^p(\g)\rightarrow \ell^{2}(\left \{H_k\right\}_{k\geq 0},\mu)}<\infty~\mathrm{with}~\mu(k)=(1+k)^{2-\frac{4}{p}-\frac{s}{p'}}\] and Example \ref{Sobolev-sharp} tells us that $\displaystyle -2+\frac{4}{p}+\frac{s}{p'}\geq 3(\frac{2}{p}-1)$ ($\Leftrightarrow s\geq p'-2$).
\end{enumerate}

\begin{remark}
The proof of Corollary \ref{cor-application} (1) is available for $O_N^+$ and $S_{N+2}^+$ with $N\geq 3$ by similar arguments. This partially reconfirms sharpness of \cite[Theorem 4.5 (2)]{Yo18a}.
\end{remark}
\end{proof}
\bibliographystyle{alpha}
\bibliography{Sobolev}

\end{document}